\documentclass[reqno,10pt]{amsart}
 
\usepackage[utf8]{inputenc}
\usepackage[T2A]{fontenc}
\usepackage[inline]{enumitem}

\usepackage{amsmath,amsthm,amsfonts, bbm, amssymb, hyperref,graphicx,color}
\usepackage{caption, subcaption}
\def\[#1\]{\begin{align}#1\end{align}}
\def\(#1\){\begin{align*}#1\end{align*}}

\newcommand{\ii}{\mathbbm{i}}

\renewcommand{\C}{\mathbb{C}}

\newcommand{\Hyp}{\mathbb{H}}

\newcommand{\R}{\mathbb{R}}
\newcommand{\Z}{\mathbb{Z}}
\newcommand{\N}{\mathbb{N}}

\newcommand{\kk}{\mathbbm{k}}
\newcommand{\jj}{\mathbbm{j}}

\newcommand{\height}{\operatorname{ht}}

\newcommand{\id}{\operatorname{id}}
\newcommand{\X}{\mathbb X}
\newcommand{\horoheight}{\operatorname{ht}}

\newcommand{\norm}[1]{\left\vert #1 \right \vert}	
\newcommand{\Norm}[1]{\left\Vert #1 \right \Vert}
\newcommand{\round}[1]{\left[#1\right]}		
\newcommand{\floor}[1]{\left\lfloor #1 \right\rfloor}

\renewcommand{\Re}{\text{Re}}
\renewcommand{\Im}{\text{Im}}

\newcommand{\Isom}{\text{Isom}}
\newcommand{\Stab}{\text{Stab}}

\newcommand{\rad}{\operatorname{rad}}

\newcommand{\comment}[1]{}

\newcommand{\ignore}[1]{}

\newtheorem{thm}{Theorem}

\newtheorem{prop}[thm]{Proposition}
\newtheorem{lemma}[thm]{Lemma}
\newtheorem{cor}[thm]{Corollary}

\numberwithin{thm}{section}

\theoremstyle{definition}

\theoremstyle{definition}
\newtheorem{example}[thm]{Example}
\newtheorem{remark}[thm]{Remark}

\numberwithin{equation}{section}

\newcommand{\Mod}{\mathcal M}

\newcommand{\Zee}{\mathcal Z}

\newcommand{\Sph}{\mathbb S}

\title[Lagrange's Theorem]{A geometric proof of Lagrange's theorem for continued fractions}
\author[A. Lukyanenko]{Anton Lukyanenko}
\address{
Department of Mathematics\\
George Mason University\\
4400 University Drive, MS: 3F2\\
Fairfax, Virginia 22030}
\email{alukyane@gmu.edu}
\author[J. Vandehey]{Joseph Vandehey}
\address{
Department of Mathematics\\
University of Texas at Tyler\\
Tyler, TX 75799
}
\email{jvandehey@uttyler.edu}

\subjclass[2020]{11K50, 11R52, 20F67  }
\keywords{Continued fractions, complex continued fractions, quaternions, octonions, Iwasawa continued fractions, Lagrange's Theorem}

\begin{document}

\begin{abstract}
    For regular continued fractions (CFs), points with finite expansions are exactly the rationals and, by Lagrange's theorem, points with eventually-periodic expansions are exactly the roots of non-degenerate quadratic equations with integer coefficients. We extend both results to proper and discrete Iwasawa CFs, including real, complex, 3D, quaternionic, octonionic, and Heisenberg CFs. Namely, the following three conditions are equivalent for a point $p$: $p$ has a finite expansion, $p\in \mathcal M(\infty)$ for the appropriate modular group $\mathcal M$, and $p$ is a fixed point of a parabolic transformation in $\mathcal M$. Eventually-periodic points correspond exactly to fixed points of \emph{loxodromic} elements of $\mathcal M$, which can be interpreted as roots of non-degenerate quadratics using the Clifford Algebra formalism of Ahlfors. In particular, this provides a new geometric proof of Lagrange's theorem for nearest-integer real CFs and Hurwitz complex CFs. Lastly, we comment on generalizations of the identity $\ii+1/\ii=0$.
\end{abstract}

\maketitle
\section{Introduction}
\label{sec:intro}

A point $x\in [0,1]$ has a finite regular continued fraction (CF) expansion exactly if it is rational, and an infinite (eventually-)periodic CF expansion exactly if it is a quadratic surd---that is, the root of a non-degenerate quadratic polynomial with integer coefficients. The former is a consequence of the Euclidean algorithm, while the latter is the content of Euler's Theorem and Lagrange's Theorem. Both results extend readily to other CF algorithms over the real numbers, including nearest-integer CFs and $\alpha$-CFs \cite{beltz2014periodicity}, as well as to certain complex CFs \cite{Dani2015}, but become more difficult for higher-dimensional CFs. Indeed, beyond the complex case, it is already non-obvious what the proper generalization of a quadratic surd should be. In \cite{vandehey2015lagrange}, the generalization for Heisenberg CFs was solutions to a class of quadratic forms.

In this paper, we will take a geometric perspective. In the reals, $x\in \hat \R$ is rational if and only if it is fixed by a non-trivial matrix $M\in SL(2,\Z)$ satisfying $|\operatorname{tr} M|=2$, called a \emph{parabolic matrix}, and acting on $\hat \R$ via M\"{o}bius transformations as follows:
\(
\begin{pmatrix}a & b \\ c & d \end{pmatrix}z= \frac{az+b}{cz+d}.
\)Similarly, $x\in \hat \R$ is a quadratic surd if and only if it is fixed point by a matrix $M\in SL(2,\Z)$ satisfying $|\operatorname{tr} M|>2$, called a \emph{loxodromic matrix}, see Theorem \ref{thm:backwards} and Remark \ref{remark:loxodromic}. In a general setting, we will show that having a finite CF expansion is equivalent to being a fixed point of a parabolic element of a generalized modular group, and having an eventually periodic CF expansion is equivalent to being a fixed point of a loxodromic element of the same group. The setting we will study is that of \emph{proper and discrete} Iwasawa CF expansions \cite{lukyanenko_vandehey_2022}, which includes as special cases the nearest-integer real CFs, $\alpha$-CFs for $\alpha\notin \{0,1\}$, and multiple infinite families of CFs over the complex numbers, $\R^3$, quaternions, and the Heisenberg group.

We will provide a full description of Iwasawa CF expansions in Section \ref{sec:background}; however, for now it suffices to know that $\X$ is an ambient space, $\Zee$ is a lattice of isometries of $\X$ that will form the digits of our CF expansion, $K\subset \X$ is a fundamental domain for $\Zee$, and $\iota$ is an inversion on $\X$. The Iwasawa CF is discrete if the modular group $\Mod=\langle\Zee, \iota\rangle$ is discrete, and it is proper if the closure of $K$ is contained in the \emph{open} unit ball centered at the origin (so that $\iota$ is a uniformly expanding map on $K$). 

The forward-shift map $T:K\to K$ acts on non-zero elements by $Tx=[\iota x]^{-1}\iota x $ where $[\iota x]\in \Zee$ is the unique element such that $[\iota x]^{-1}\iota x\in K$. As such, the action of $T$ on $x$ can be represented by an element of $\Mod$ in a natural way. Note that in most contexts we think of $\Zee$ as being additive, and so write this instead as $Tx= \iota x - [\iota x]$. The CF expansion of $x\in K$ is, then, the sequence of \emph{digits} $a_i = [\iota T^{i-1} x]$. If the event that $T^i x=0$, the digit $a_{i}$ and subsequent digits are undefined and the expansion is finite. If the sequence of $a_i$'s is eventually periodic, which is equivalent to $T^i x = T^j x$ for some $i\neq j$, then we say that the expansion is eventually periodic.

Our main result is as follows:

\begin{thm}
\label{thm:generalmain}
Let $(\X, \Zee, \iota, K)$ be a proper and discrete Iwasawa CF. Then $x\in X$ has a finite expansion if and only if it is a fixed point of a parabolic element of $\Mod=\langle\Zee, \iota\rangle$, and has an eventually-periodic expansion if and only if it is a fixed point of a loxodromic element of $\Mod$.
\end{thm}

\begin{example}
Consider the data $\X=\R$, $\Zee=\Z$, $\iota(x)=-1/x$, and $K=[-1/2,1/2)$, giving the backwards nearest-integer CFs. The associated modular group $\Mod$ is the familiar $SL(2,\Z)$, since both translations and the inversion can be written as matrices in $SL(2,\Z)$. This CF is \emph{discrete} because $\Mod$ is a discrete subgroup of $SL(2,\R)=\Isom(\Hyp^2_\R)$. The CF is \emph{proper} because the closure of the region $K$ is contained in the open interval $(-1,1)$, so that $\iota$ is a uniformly expanding mapping on $K$.
\end{example}

Theorem \ref{thm:generalmain} fails without the properness assumption:
\begin{example}
For negative variants of even CFs (that is, $\X=\R$, $\iota(z)=-1/z$, $\Zee = 2\Z$) and the J.Hurwitz CFs (that is, $\X=\C$, $\iota(z)=-1/z$, $\Zee = \langle 1\pm \ii\rangle $), one has $1\notin \Gamma(\infty)$. Furthermore, $-1=[\overline 2]$, a repeating expansion giving a parabolic symmetry $x\mapsto \frac{-1}{2+x}$.
\end{example}

\begin{example}
The phenomenon of periodic expansions for non-loxodromic fixed points was studied by Schmidt and Sheingorn \cite{schmidt1995length} for a variant of the Rosen continued fractions called the $n\lambda F$ continued fractions, given by $\X=\R$, $\iota(x)=-1/x$, $\Zee = \mathbb{Z}\lambda_q$, and $K=(0,\lambda_q]$.  They show that a fixed point of a parabolic element either has a finite expansion or is periodic with period $[2\lambda_q,\underbrace{\lambda_q,\dots \lambda_q}_{q-1\text{ times}}]$. All other periodic expansions are fixed points of loxodromic matrices (and vice-versa).
\end{example}

Further examples, including a discussion of the relationship between quadratic surds and fixed points of loxodromic matrices, can be found in Section \ref{sec:loxodromic}.

\subsection{The real case}
Let us now sketch the proof of Theorem \ref{thm:generalmain} in the case of nearest-integer CFs, without invoking the full machinery of Iwasawa CFs. We will comment on places in the proof that become more complicated in the general case. We will then comment on the comparison to the classical algebraic proof in Remark \ref{remark:comparison}.

\begin{proof}[Sketch of the proof of Theorem \ref{thm:generalmain} for nearest-integer CFs]
\newcommand{\parabolicmatrix}{\begin{pmatrix}1&1\\0&1\end{pmatrix}}
To see that finite CFs are parabolic fixed points, observe first that $\infty$ is a fixed point of the parabolic transformation $z\mapsto z+1$ associated to the matrix $\parabolicmatrix$. If $x$ has a finite CF expansion, then $T^n x=0$ for some $n$. Write $\iota T^n$ as a matrix $M\in SL(2,\Z)$, and observe that $x$ is a fixed point for $M^{-1}\parabolicmatrix M$, which is a conjugate of a parabolic matrix and hence also parabolic. This step generalizes immediately to the general case.

To see that parabolic fixed points have finite expansions, observe that such fixed points are solutions to degenerate quadratic polynomials, i.e., rational points. Now, associate to each rational point $p/q$ (in lowest terms) the associated Ford circle in the complex upper-half plane with radius $1/q^2$ tangent to the point $p/q$ on the real axis, including the ``circle at infinity'' consisting of the region $\{x+\ii y: y=1\}$. Working with the generators $z\mapsto z+a$ and $\iota$ of $\Mod$, one shows that the collection of all Ford circles is invariant under the associated modular group $\Mod=SL(2,\Z)$. Now, given a rational point $p/q\in (-1/2,1/2]$, consider the action of the map $T$ on the Ford circle sitting at $p/q$. One shows that inversion makes the circle bigger by a definite amount (properness is critical here), and translation preserves the radius of the circle.  Thus, iterating the map $T$ moves $p/q$ while also expanding its associated Ford circle by a definite multiplier. Because finite circles have a bounded diameter, they must eventually become the circle at infinity, and so the rational point must reach $\infty$. In the full proof of Theorem \ref{thm:generalmain}, we obtain generalized Ford circles from the fact that $\Mod$ is a lattice, and therefore has a horoball corresponding to every fixed point of a parabolic isometry. 

Next, consider an eventually-periodic point $x\in K=(-1/2,1/2]$, and assume for simplicity that it is purely periodic with period $n$. Then the mapping $T^n$ acting on $x$ is associated with a matrix $M\in SL(2,\Z)$. One concludes that $M$ is loxodromic because it is expanding near $x$: indeed, the generalized Gauss map is a composition of (isometric) translations and inversions $\iota$ that are expanding near points in $K$. Note that over $\R$, loxodromic fixed points correspond exactly to solutions of non-degenerate integer quadratics, Theorem \ref{thm:backwards}.

The last part of the proof relies on a lemma about the action of the modular group on hyperbolic space: for $K_1, K_2 \subset \Hyp$ compact, the set $\{g\in \Mod: gK_1\cap K_2 \neq \emptyset\}$ is finite. To prove this, identify a point $x+\ii y \in \Hyp$ with the matrix $f(x+\ii y)=\begin{pmatrix} y &x \\ 0 &1 \end{pmatrix}\in SL(2,\R)$ which sends $\ii$ to $x+\ii y$. Noting that the stabilizer of $\ii \in \Hyp$ is compact, we have that the  set
$$\{g_1 g_2 g_3^{-1}: g_1 \in f(K_2), g_2 \in \Stab_{SL(2,\R)}(\ii), g_3\in  f(K_1)\}$$ 
is compact and consists of all isometries sending some point of $K_1$ to some point of $K_2$. Since the modular group $\Mod=SL(2,\Z)$ is discrete, it has only finitely many points in this compact set, as desired.

\begin{figure}[ht]
\caption{Geometric setup for the proof of Lagrange's theorem for nearest-integer CFs, under the simplifying assumption that $(x,x')$ are widely spaced, i.e. $|x|<0.5$ and $|x'|>\sqrt{2}$.}
\label{fig:keyargument}
\includegraphics[width=\textwidth]{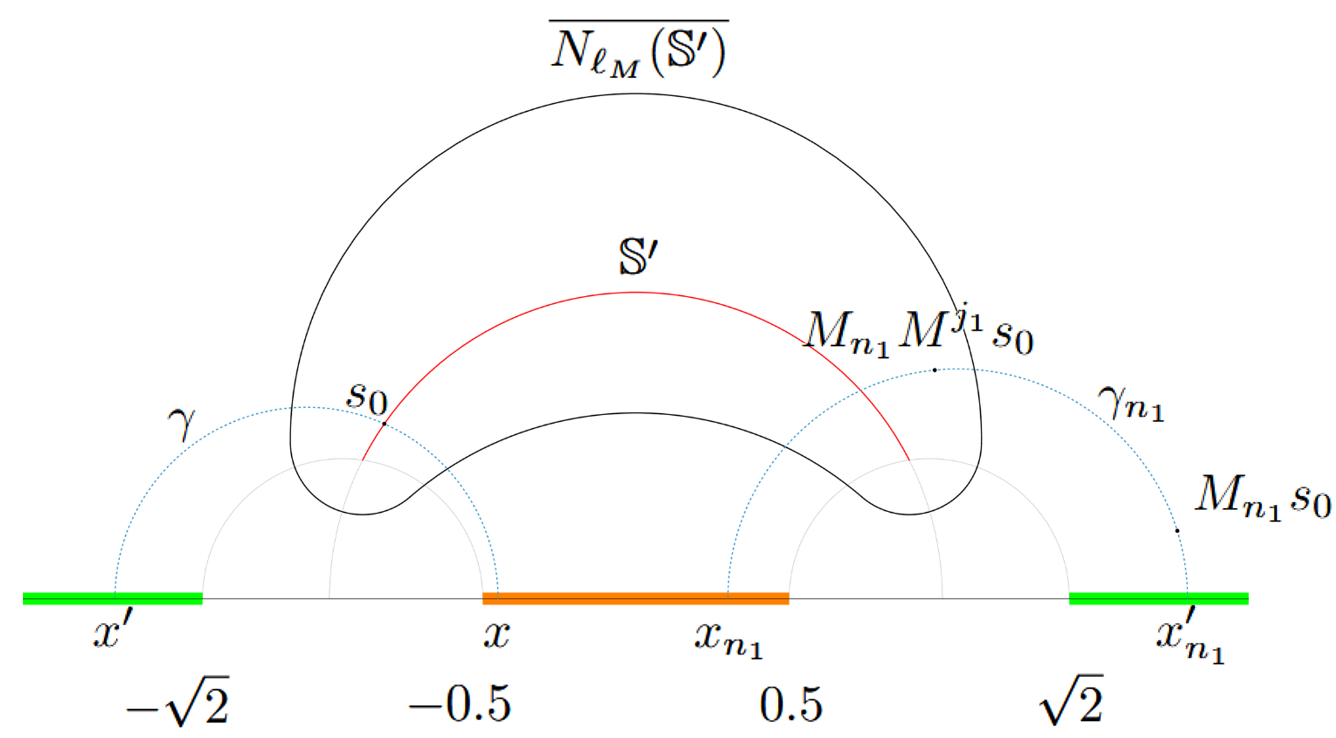}
\end{figure}

We will also need a geometric construction, see Figure \ref{fig:keyargument}. For the purposes of this proof, we will call a pair of points $(a,b)$ \emph{widely-spaced} if $|a|<1/2$ and $|b|>\sqrt{2}$. The geodesic joining these points in hyperbolic space passes through the unit circle at a point whose imaginary component is bounded below by $C=(1/7)\sqrt{3(9-4\sqrt{2})}>0$. Let $\Sph'=\{z\in \Sph : \Im(z)\geq C\}$. For any $\epsilon>0$, the closed hyperbolic neighborhood $\overline{N_\epsilon(\Sph')}$ of $\Sph'$ remains compact.

Now, let $M$ be a loxodromic matrix with distinct fixed points $x,x'\in \hat\R$ joined by a hyperbolic geodesic $\gamma$. Then, $M$ moves along $\gamma$ by some distance $\ell_M$. We assume that $x\in K$, and observe that $x$ has an infinite forward orbit under $T$: as we have shown, a finite orbit would imply that $x=M'\infty$ for some $M'\in SL(2,\Z)$, but then $M'MM'^{-1}$ would be a loxodromic element in $SL(2,\Z)$ fixing $\infty$, which is impossible since $M'MM'^{-1}\infty=\infty$ and $\det M'MM'^{-1}=1$ imply $|\operatorname{tr} M'MM'^{-1}|=2$. 

Let $M_n$ be the matrix associated to the $n$th iterate of the generalized Gauss map at $x$, so that $M_nx=T^n(x)$, and let $x_n = M_n x$, $x'_n = M_n x'$, and $\gamma_n=M_n\gamma$.  We first claim that, for some $n$ one has that $|x'_n|>\sqrt{2}$, so that the pair $(x,x')$ is widely spaced. For this, suppose that $|x'|<\sqrt{2}$ and let $a$ be the first digit of $x$. Using the inversion identity $|x||y||1/x-1/y|=|x-y|$ and the fact that $|x|\leq 1/2$, we have that $|x_1-x'_1|=|M_1(x)-M_1(x')| = |(1/x-a) - (1/x'-a)| = |x|^{-1}|x'|^{-1}|x-x'| \geq \sqrt{2} |x-x'|$.

Under further iteration, the points continue to repel each other until $x'$ is pushed out of the interval $(-\sqrt{2},\sqrt{2})$. Further iteration may bring the iterates of $x, x'$ back together, but it must eventually push them apart again, giving an infinite sequence $n_i$ such that for each $i\geq 0$ we have $|x'_{n_i}|>\sqrt{2}$. We assume, for simplicity, that $n_1=0$ and let $s_0=\gamma\cap \Sph \in \Sph'$. Now, for $i>1$, the point $M_{n_i}s_0$ may be far from $\Sph$, but we can move it along $\gamma_i$ using powers of our loxodromic mapping $M$, so that for some $j_i$ we have that $M_{n_i}M^{j_i}s_0\in \overline{N_{\ell_M}(\Sph')}$. We thus obtain infinitely many matrices $M_{n_i}M^{j_i}$ sending $s_0$ into $\overline{N_{\ell_M}(\Sph')}$. Because of the geometric lemma above, only finitely many such matrices exist, so that there are indices $i_1 \neq i_2$ such that $M_{n_{i_1}}M^{j_{i_1}}=M_{n_{i_2}}M^{j_{i_2}}$. Since $M$ fixes $x$, this implies $M_{n_{i_1}}x=M_{n_{i_2}}x$ so that the orbit of $x$ under the generalized Gauss map is eventually periodic, and thus the digits of $x$ are also eventually periodic, as desired.
\end{proof}

\begin{remark}
\label{remark:comparison}
The most difficult part of the above work is the proof of Lagrange's Theorem, which says that any quadratic surd has an eventually periodic CF expansion. There are many different proofs of this theorem already:
\begin{enumerate}
\item In one method (see \cite{northshield2011short}), the minimal integer polynomials for $T^n x$ are proven to all have the same discriminant. The coefficients of these polynomials are proven to be bounded, showing that at some point the polynomials (and hence the values of $T^n x$) must repeat. This is also the method used by \cite{Dani2015}.
\item Another method (see \cite[Ch.~2]{borwein2014neverending}) uses the same first step as the previous method, but then proves that there are finitely many elements in $\mathbb{Q}(\sqrt{d})$ whose minimal polynomial has a given discriminant.
\item In a third method (see \cite[Thm.~1.2]{HensleyBook}), it is shown that each $T^n x$ can be put into the form $\frac{a+\sqrt{d}}{b}$ where $b|(a^2-d)$. This relation restricts the possible values of $T^n x$ to a finite set.
\end{enumerate}
Many of the methods above implicitly make use of the fact that $\Z$ is an ordered discrete set to prove finiteness, and even Dani's work over the complex numbers uses that $\C$ is a commutative ring with a multiplicative norm that respects the triangle inequality. The method given in this paper neatly avoids requiring any assumptions about ring structure or order because finiteness is a direct result of the discreteness of the modular group. As a result, our results extend to spaces without such structures, including $\R^3$ and the Heisenberg group.
\end{remark}

\subsection{Organization}
We define Iwasawa CFs and describe the associated geometry in \S  \ref{sec:Iwasawa}. We then prove the main theorem in \S \ref{sec:mainproof}. We then focus on characterizing fixed points of loxodromic elements in \S \ref{sec:loxodromic}. In \ref{sec:backwardscompatability}, we show that loxodromic fixed points for $\R$ and $\C$ are exactly the solutions of non-degenerate quadratics. We then provide the framework in Clifford matrices in \S \ref{sec:Clifford}, with applications to 3D and 4D (quaternionic) CFs in \S  \ref{sec:R3} and \ref{sec:R4}, respectively. We finish in \S \ref{sec:identities} by remarking on some curious identities for quaternions, generalizing the continued fraction identity $\ii+1/\ii=0$.

\subsection{Acknowledgements}
The research was sponsored in part by the first author's Simons Foundation grant MPS-TSM: ``Geometry and dynamics of higher-dimensional continued fractions''.

\section{Geometry of Iwasawa CFs}\label{sec:background}
\label{sec:Iwasawa}

We work with the Iwasawa continued fraction framework \cite{lukyanenko_vandehey_2022}, defined as follows. 

Fix a real associative division algebra $k\in \{\R, \C, \text{Quaternions}\}$ and $n\in \mathbb N$. The Iwasawa inversion space $\X=\X^n_k$ is modeled on $k^{n}\times \Im(k)$ with group law $(z,t)*(z',t') = (z+z', t+t'+2\Im (\overline z\cdot z'))$ with identity $0=(0,0)$ and group inverse given by $(z,t)*(-z,-t)=(0,0)$. One gives $\X$ a left-invariant metric by choosing the gauge $\norm{(z,t)} = \sqrt[4]{\Norm{z}^4 + \Norm{t}^2}$ and setting $d((z,t),(z',t')) = \norm{(-z,-t)*(z',t')}$. The natural inversion on $\X$ is the \emph{Koranyi inversion}
$\iota_-(z,t)=\left(\frac{-z}{\Norm{z}^2+t},\frac{-t}{\Norm{\Norm{z}^2+t}^2}\right)$,
satisfying the identities
\begin{equation}\label{eq:magicinversion}\norm{\iota_-(p)}=\norm{p}^{-1},\hspace{1in}d(\iota_- p, \iota_- p') = \norm{p}^{-1}\norm{p'}^{-1}d(p,p').\end{equation}

The associated hyperbolic space $\Hyp$ (a negatively-curved rank-one symmetric space) in modeled in horospherical coordinates as $\X\times \R_+$. Identifying $\R$ with $\Re(k)$, we view a point in $\Hyp$ as a triple $(z,t,s)$ where $z\in k^n$, $t\in \Im k$, and $s\in \Re(k)$, furthermore writing $w=s+t\in k$ so that $(z,t,s)=(z,w)$ becomes a point in the upper half-space in $k^{n+1}$. We will refer to $\operatorname{ht}_\infty(z,t,s)=s$ as the \emph{horoheight\footnote{Due to different normalizations of the metric in the case $k=\R$ and $k\neq \R$, we have that for $k=\R$ the function $\operatorname{ht}_\infty$ is in fact the square of the horoheight from infinity, cf.~Remark \ref{remark:realHyperbolic}.} from infinity}. 

Consider the following group actions: an action of the group $\X$ via $(p,s)\mapsto (p_0*p, s)$, an action of $\R_+$ via $(p,s)\mapsto(\delta_rp, r^2 s)$ where $\delta_r(z,t)=(rz,r^2t)$, and the order-two extended Koranyi inversion $\iota_-(z,w) = \left(\frac{-z}{\Norm{z}^2+w}, \frac{\overline w}{\Norm{\Norm{z}^2+w}^2}\right)$. The gauge metric extends to $\Hyp$ via the gauge $\norm{(z,w)}=\Norm{\Norm{z^2}+w}^{1/2}$ and invariance under the action of $\X$, giving the Cygan metric $d_C$ on $\overline \Hyp$ that generalizes the Euclidean metric on the ambient space. The negatively-curved Riemannian metric on $\Hyp$, which we don't define here, is invariant under the actions of $\X$, $\R_+$, and the Koranyi inversion. According to the Iwasawa (or $KAN$) decomposition, the isometry group of $\Hyp$ can be written as the direct sum $\Isom(\Hyp)=KAN$, where $N=\X$, $A=\R_+$, and $K$ is a compact Lie group fixing the point $(0,1)\in \Hyp$. Note that we will use the variables $K, A, N$ in other ways when we are not talking about the Iwasawa decomposition, and hope that the use is clear from context.

\begin{remark}\label{remark:realHyperbolic}The simplest case is $k=\R$ and $n=2$, giving $\X=\R$ with the usual metric and an unusual model of the hyperbolic upper-half plane. Namely, $\Hyp$ is $\R\times\R_+$ with coordinates $(z,w)$, the extended Cygan metric characterized by $\norm{(z,w)}=\sqrt{z^2+w}$, and Koranyi inversion $\iota_-(z,w)=\left(\frac{-z}{z^2+w}, \frac{w}{z^2+w}\right)$. To convert to the Poincar\'e upper half-plane with the Euclidean metric and inversion $x+\ii y\mapsto \frac{-1}{x+\ii y}$, one takes $x+\ii y=z+\ii \sqrt{w}$.
\end{remark}

We will need the following well-known lemma concerning the action of discrete groups of isometries of $\Hyp$; we sketch the proof for completeness. 
\begin{lemma}
\label{lemma:properaction}
For $K_1, K_2 \subset \Hyp$ compact and $\mathcal M$ a discrete group of isometries of $\Hyp$, the set $\{g\in \Mod: gK_1\cap K_2 \neq \emptyset\}$ is finite.
\end{lemma}
\begin{proof}
The group $U(k, n,1)=\Isom(\Hyp)$ acting on $\Hyp=\Hyp^n_k$ has an Iwasawa decomposition $U(k, n,1)=KAN$. The groups $A$ and $N$ are described above: $A$ is isomorphic to $\R_+$ and acts on horospherical coordinates by generalized dilations, while $N$ is isomorphic to $\X$ and acts on horospherical coordinates by left multiplication in the first coordinate. The group $AN$ acts simply transitively on $\Hyp$. This gives a bijection from $AN$ to $\Hyp$ given by $an\mapsto an(0,1)$, whose inverse is then a diffeomorphism $f: \X \rightarrow AN$. Lastly, the group $K$ is isomorphic to the compact group $U(k,n-1)$, and accounts for rotations around the point $(0,1)$. One concludes that the set
$$\{g_1g_2g_3^{-1}: g_1\in f(K_2), g_2\in U(k,n-1), g_3\in f(K_1)\}$$
is compact, i.e., the action of $U(k,n,1)$ on $\Hyp$ is proper. Since $\Mod$ is discrete, it intersects the above set in a finite number of points, as desired.
\end{proof}

We will also need to see how geodesics in $\Hyp$ intersect the sphere $\Sph=\{(z,w):|(z,w)|=1\}$. We will use the following lemma to obtain the set $\Sph'$, see Figure \ref{fig:keyargument}.
\begin{lemma}
\label{lemma:horoheightminimum}
Let $0<\epsilon < 1 < \epsilon '$. Then there exists $h_0>0$ such that any geodesic $\gamma$ with endpoints $(a,b)$ satisfying $\norm{a}<\epsilon$ and $\norm{b}>\epsilon'$ intersects the gauge sphere $\Sph$ at points whose horoheight from infinity is bounded below by $h_0$.
\begin{proof}
The existence of the intersection with the gauge unit sphere is guaranteed by the intermediate value theorem.

Fix a smooth parametrization $\gamma_0: [0,\infty]\rightarrow \overline{\Hyp}$ of the geodesic joining the points $0$,$\infty$, e.g. $\gamma_0(t)=(0,0,t)$. We may vary the first endpoint inside the compact set $\overline{B(0,\epsilon)}$ by applying elements of the nilpotent group $N$. Likewise, we may move the second point by applying elements of $\iota_- N \iota_-$ where $\iota_-$ is the Koranyi inversion.  All together, this gives us a smooth mapping $\phi: [0,\infty]\times \overline{B(0,\epsilon)} \times \overline{(\hat \X\setminus B(0,\epsilon'))}\rightarrow \Hyp$ that parametrizes all points on all geodesics of interest, including ones on the boundary. Letting $\psi(t,a,b)=\norm{\phi(t,a,b)}$, we have that $\psi^{-1}(1)$ contains all the intersections with the sphere. This is a closed set, and therefore compact.

Now, the horoheight from infinity of $\psi(t,a,b)$ is equal to 0 if and only if $t\in \{0,\infty\}$, but geodesic endpoints lie in $B(0,\epsilon)$ or $\hat X \setminus B(0,\epsilon')$, neither of which intersect $\Sph$. Thus, horoheight is in fact positive for every point of $\psi^{-1}(1)$, and attains a minimum on $\psi^{-1}(1)$.
\end{proof}
\end{lemma}

Now, to define a particular continued fraction algorithm on $\X$, fix a discrete group $\Zee \subset \Isom(\X)$ acting on both $\X$ and $\Hyp$, a fundamental domain $K\subset \X$ for $\Zee$, contained in the closed unit ball, and an inversion $\iota = \mathcal O \circ \iota_-$. 
Based on the data $(k, n, \Zee, K, \iota)$, define:
\begin{enumerate}
\item The rounding function $\round{\cdot}_K: \X\rightarrow \Zee$ characterized by $\round{p}_K^{-1}(p)\in K$,
\item The generalized Gauss map $T: K\setminus \{0\}\rightarrow K$ defined by $T(p) = \round{\iota p}^{-1}(\iota p)$,
\item The CF digits $\{a_i(p)\}$ of a point $p\in K$, given by $a_i(p)=[\iota T^{i-1}p]$, ending with $i=n$ if $T^n p=0$,
\item Associated to a sequence of $a_i\subset \Zee$, the CF $[a_1,a_2,a_3,\dots, a_n]=a_1 \iota a_2 \iota \cdots a_n(0)$, taking a limit for infinite sequences,
\item The generalized modular group $\Mod=\langle \Zee,\iota\rangle$,
\item The radius of $K$: $\rad(K)=\sup\{\norm{p}: p\in K\}$,
\item The \emph{goalpost region} $\tilde K\subset \X\times \R_+=\Hyp$ consisting of points $(p,s)$ satisfying $p\in K$ and $d_C((p,s),(0,0))\geq 1$.
\end{enumerate}

\begin{example}
Regular CFs are recorded via $\X=\R, \iota(x)=1/x, K=[0,1)$ giving $[x]_K=\floor{x}$, and $\Zee=\Z$, with associated hyperbolic space $\Hyp=\Hyp^2_\R$ the real hyperbolic plane and  modular group $\Mod\cong SL(2,\Z)$. 
\end{example}
\begin{example}Hurwitz complex CFs are recorded with $\X=\R^2, \iota(x)=\frac{(x,-y)}{x^2+y^2}, K=[-1/2,1/2)\times [-1/2,1/2)$, and $\Zee=\Z[\ii]$, with associated hyperbolic space $\Hyp=\Hyp^3_\R$ and generalized modular group $\Mod\cong SL(2,\Z[\ii])$.
\end{example}

An Iwasawa CF is \emph{proper} if $\rad(K)<1$, in which case $\iota\vert_K$ is uniformly expanding.
The CF is \emph{discrete} if $\Mod$ is discrete. For discrete and proper Iwasawa CFs, the generalized modular group $\Mod$ is a non-uniform lattice, i.e., the quotient modular manifold $\Mod\backslash\Hyp$ has finite volume but is not compact. Indeed, one shows (see \cite{lukyanenko_vandehey_2022}) that any point in $\Hyp$ can be moved into $\tilde K$ using a sequence of translations (which don't affect $s$) and inversions (which increase $s$ when $p\in K$). In the case of nearest-integer real CFs, the goalpost region $\tilde K$ already gives the familiar fundamental domain for $SL(2,\Z)$. More generally, $\tilde K$ only \emph{contains} a fundamental domain for $\Mod$ at a finite index, see the discussion of hidden symmetries in \cite{lukyanenko_vandehey_2022}. Nonetheless, the properness assumption implies that $K$ has finite volume. 

For a point $(p,s)\in \X\times \R_+ = \Hyp$, one defines the horoheight from infinity as $\horoheight_\infty(p,s)=s$. Sets of the form $\{(p,s): s=s_0\}$ are \emph{horospheres at $\infty$}, and the sets of the form $\{(p,s): s\geq s_0\}$ are \emph{horoballs at $\infty$}. Horospheres (resp., horoballs) at other points are images of horospheres (resp., horoballs) under an arbitrary $g\in \Isom(\Hyp)$, and touch the boundary at the point $g(\infty)$. In particular, for a horoball $\mathcal B$, one has that $\sup \{\horoheight_\infty(p,s): (p,s)\in \mathcal B\}$ is infinite if and only if $\mathcal B$ is based at $\infty$. For $\Hyp=\Hyp^2_\R$, horospheres are circles or horizontal lines, and $\Gamma$-invariant collections of horospheres generalize Ford circles.

\begin{remark}
The stabilizer of infinity $\Stab_{\Isom(\Hyp)}(\infty)$ sends horospheres at $\infty$ to horospheres at $\infty$ and horoballs at $\infty$ to horoballs at $\infty$, so that one may speak intrinsically of horospheres and horoballs, but not of their height. In \cite{lukyanenko_vandehey_2022}, we were interested only in horoballs based at points of the form $g(\infty)$ for $g\in \Mod$. In that setting, $\Stab_\Mod(\infty)$ consists only of transformations that preserve horoheight at $\infty$, so one may define, for $g\in \Mod$, $\horoheight_{g(\infty)} (p,s) := \horoheight_\infty(g^{-1}(p,s))$. We will not use this definition in this paper.
\end{remark}

An isometry $g$ of a Gromov hyperbolic space $H$ falls into one of four mutually-exclusive types, see e.g.~\cite{DasSimmonsUrbanski}:
\begin{enumerate}
\item The identity.
\item Elliptic, if it is not the identity and has a fixed point in $H$. Any fixed points on the boundary are neither attracting nor repelling.
\item Parabolic, if it has exactly one fixed point in $\partial H$, which is neither attracting nor repelling,
\item Loxodromic, if it has exactly two fixed points in $\partial H$, one of which is attracting and the other repelling.
\end{enumerate}

For $\Hyp=\Hyp^{n+1}_k$ viewed in horospherical coordinates, one  can further normalize the parabolic and loxodromic isometries by conjugating by an element of $\Isom(\Hyp)$. A mapping $g$ is:
\begin{enumerate}
\item Parabolic if and only if it is equivalent to a mapping of the form $(p,s)\mapsto (p_0,0)*(p,s)$, with $p_0\neq 0$,
\item Loxodromic if and only if it is equivalent to a mapping of the form $(p,s)\mapsto O(\delta_r(p,s))$, where $\delta_r$ is a dilation by factor $r>0$ and $O$ is a rotation of $\X$ acting only on the first coordinate,
\item Hyperbolic if it is loxodromic and $O=\id$.
\end{enumerate}
In particular, a loxodromic isometry moves along the geodesic joining its fixed points on the boundary, while also twisting around it if the isometry is non-hyperbolic.

\begin{remark}
\label{remark:loxodromic}
For matrices in $SL(2,\R)=\Isom^+(\Hyp^2_\R)$, the conditions \emph{hyperbolic} and \emph{loxodromic} coincide, and \emph{hyperbolic} is the more familiar term in that context.
\end{remark}

\section{Proof of the Main Theorem}
\label{sec:mainproof}
We prove Theorem \ref{thm:generalmain} in two parts. First, we characterize the finite CFs:
\begin{lemma}
\label{lemma:rationalPoints}
Fix a proper real Iwasawa CF. Then the following are equivalent for a point $p\in \X$:
\begin{enumerate}
\item $p$ is a fixed point of a parabolic element in $\Mod$, 
\item $p\in \Mod(\infty)$,
\item $p$ has finitely many CF digits.
\end{enumerate}
\end{lemma}

To understand an arbitrary parabolic element of $\Mod$, we will need the theory of geometrically finite groups. In the special case of $\R^n$ including real, complex, and quaternionic CFs, the primary reference is Ratcliffe's book \cite{Ratcliffe}. In our broader context that includes Heisenberg CFs associated to complex hyperbolic space, we follow Bowditch \cite{Bowditch}.

\begin{lemma}Suppose $\Gamma\subset \Isom(\Hyp^n_k)$ is a lattice. Then it is geometrically finite.
\begin{proof}
Bowdich \cite{Bowditch} provides four equivalent  conditions characterizing geometric finiteness in pinched curvature, which includes the rank one symmetric spaces $\Hyp^n_k$. Condition F5 states that there is a bound on the orders of every finite subgroup of $\Gamma$, and also some neighborhood of the convex core of $\Gamma \backslash \X$ has finite volume. The first condition is redundant for lattices by Proposition 5.4.2 of \cite{Bowditch}. Since $\Gamma$ is a lattice, all of $\Gamma \backslash \X$ has finite volume (and the convex core coincides with $\Gamma \backslash \X$), so the second condition is also satisfied. 
\end{proof}
\end{lemma}

In geometrically finite groups, parabolic elements can be extended to large subgroups. Adapting Lemma 6.4 of \cite{Bowditch}, we have:
\begin{lemma}
Suppose $\Gamma\subset \Isom(\Hyp)$ is geometrically finite and $p\in \hat \X=\X\cup\{\infty\}$ a fixed point of a parabolic element. Then $p$ is a bounded-parabolic fixed point, i.e., $\Stab_\Gamma(p)\backslash (\hat \X\setminus\{p\})$ is compact.
\end{lemma}

Lemma 6.2 of \cite{Bowditch} associates a standard cusped region to each bounded-parabolic point in the limit set of a discrete group. In our setting, all parabolic points are bounded-parabolic and standard cusped regions are horoballs, so the result gives us a generalized Ford circle at each parabolic fixed point:
\begin{lemma}
\label{lemma:horoballs}
Suppose $\Gamma \subset \Isom(\Hyp)$ is a lattice. Let $\Pi\subset \hat \X$ be the collection of parabolic fixed points of $\Gamma$. Then there is a collection of disjoint horoballs $B=\{\mathcal B_p: p\in \Pi\}$ that is invariant under $\Gamma$, i.e., for each $\gamma\in \Gamma$ and $\mathcal B\in B$ we have that $\gamma(\mathcal B)\in B$.
\end{lemma}

As one consequence of Lemma \ref{lemma:horoballs}, it makes sense to distinguish parabolic fixed points in $\hat \X$ from  loxodromic fixed points in $\hat \X$:
\begin{cor}
\label{cor:LoxodromicIsNotParabolic}
Suppose $\Gamma \subset \Isom(\Hyp)$ is a lattice. If $p\in \X$ is the fixed point of a parabolic element of $\Gamma$, then it is not a fixed point of a loxodromic element of $\Gamma$.
\begin{proof}
Since $\Isom(\Hyp)$ acts transitively on $\hat \X$, we may assume that $p=\infty$ and the associated horoball $\mathbb B$ is defined by $\height_\infty(z,t,s)=s\geq s_0$ for some $s_0$. If $p$ is fixed by a loxodromic element $\gamma$ preserving $\mathbb B$, it would have a second fixed point in $\X$, which we may normalize to $(0,0)\in \X$. But then it would preserve the point $(0,0,s_0)$, meaning that $\gamma$ is in fact parabolic, not loxodromic, a contradiction.
\end{proof}
\end{cor}

With these results in hand, we can prove Lemma \ref{lemma:rationalPoints}:
\begin{proof}[Proof of Lemma \ref{lemma:rationalPoints}]
(3$\rightarrow$2) If $p$ has finitely many digits, say $\{a_i(p)\}_{i=1}^n$, then by definition $p=a_1\iota a_2\iota \dots a_n (0)$ and hence $p = a_1\iota a_2 \iota \dots a_n \iota (\infty)$. Thus $p\in \Mod(\infty)$. 

(2$\rightarrow$1) If $p\in \Mod(\infty)$, then $p=\gamma\infty$ for some $\gamma\in \Mod$, and so it is fixed by every parabolic element $\gamma a\gamma^{-1}$ for $a\in \Zee\setminus \id$.

(1$\rightarrow$3) Let $B=\{\mathcal B_p\}$ be the collection of disjoint horoballs given by Lemma \ref{lemma:horoballs}, indexed by fixed points of parabolic elements of $\Mod$. In particular,  since $\Zee$ contains non-trivial parabolic elements, the point $\infty\in \hat \X$ is a parabolic fixed point, and $B$ contains a horoball at infinity of some height $h_0$: $\mathcal B_\infty = \{(p,s)\in \X\times \R_+: s>h_0\}$. Now, suppose $p\in \X$ is another parabolic fixed point with associated horoball $\mathcal B_p$. We may normalize $p$ by applying $\round{p}_K^{-1}$ so that $p\in K$. Recall that for a set $A\subset \Hyp$, the horoheight from $\infty$ is given by $\horoheight_\infty(A)=\sup \{s : (p,s)\in A\}$. Our normalization clearly does not affect the horoheight of $\mathcal B_p$. If we now have $p=0$, then we are done. Otherwise, applying $\iota$ produces a new horoball $\mathcal B_{\iota p}$, whose horoheight, by Lemma 3.8 of \cite{lvDiophantine}, satisfies $\horoheight_\infty(\iota(\mathcal B_p)) = |p|^{-1} \horoheight_\infty(\mathcal B_p) > (\rad K)^{-1}\horoheight_\infty(\mathcal B_p)$. Repeating the process of normalization and inversion for $i$ steps produces a horoball whose horoheight from infinity is at least $(\rad K)^{-i}\horoheight_\infty(\mathcal B_p)$. If the process were to repeat infinitely, this would exceed the horoheight of $\mathcal B_\infty$, causing a non-trivial intersection between them, contradicting disjointness. Since each step of inversion followed by normalization acts via the map $T$, we see that after finitely many steps we must have $T^ip = 0$, as desired.
\end{proof}

For the second half of Theorem \ref{thm:generalmain}, we prove the generalized Euler-Lagrange theorem:

\begin{lemma}Fix a proper and discrete Iwasawa CF. Then $x\in \X$ has an eventually-periodic expansion if and only if it is the fixed point of a loxodromic isometry of the associated hyperbolic space $\Hyp$.
\begin{proof}
Suppose first, without loss of generality, that $x\in K$ is purely periodic, so $T^n x= x$ for some $n>0$. That is, $a_n^{-1} \iota a^{-1}_{n-1} \iota \cdots a_1^{-1} \iota x = x$, for digits $a_1, \ldots, a_n\in \Zee \subset \Isom(\X)$. Extending the mappings to isometries of hyperbolic space, we have $M_n x = x$ for $M_n\in \Isom(\Hyp)$. To prove that $M_n$ is loxodromic, observe that $x$ is a repelling fixed point: for a nearby point $y$, by the inversion identities \eqref{eq:magicinversion}, we have 
$$d(\iota x, \iota y) = \norm{x}^{-1}\norm{y}^{-1} d(x,y) > \rad(K)^{-1} \norm{y}^{-1} d(x,y).$$
Since $a_1$ is an isometry,
$$d(a_1^{-1} \iota x, a_1^{-1}\iota y) > \rad(K)^{-1} \norm{y}^{-1}  d(x,y).$$
Repeating, one obtains
$$d(M_n x, M_n y) \geq \rad(K)^{-n}\norm{M_{n-1} y}^{-1}\cdots \norm{y}^{-1} d(x,y)$$
We may choose $y$ sufficiently close to $x$ so that for each $m<n$ the iterates $M_my$ satisfy $\norm{M_my}<1$, giving $d(M_nx, M_ny)\geq \rad(K)^{-n}d(x,y)$. This shows that $x$ is a repelling fixed point of $M_n$, as desired. 

Assume, conversely, that $M\in \Mod$ is loxodromic with distinct fixed points $x, x'$, with $x\in K$, joined by a hyperbolic geodesic $\gamma$.  Fix the sequence mappings $M_n\in \Mod$ associated to the generalized Gauss map for $x$, so that $T^n x = M_n x$. Set $x_n = M_n x$, $x'_n = M_n x'$, and $\gamma_n = M_n \gamma$. In view of Corollary \ref{cor:LoxodromicIsNotParabolic} and Lemma \ref{lemma:rationalPoints}, this sequence is infinite.

We will say that a pair of points $(a,b)$ is \emph{widely-spaced} if $\norm{a}<\rad(K)$ and $\norm{b}>\rad(K)^{-1/2}$. We claim that there is an $n_0>0$ such that the pair $(x_{n_0}, x'_{n_0})$ is widely spaced. Indeed, if this condition is not satisfied for the first $n$ iterations, then by \eqref{eq:magicinversion} we have $d(M_n(x), M_n(x'))\geq \rad(K)^{-n/2}d(x,x')$, which leads to a contradiction since the points are forced arbitrarily far apart as $n$ increases, with $M_n(x)$ remaining in $K$. When $(x_n, x'_n)$ is widely spaced, Lemma \ref{lemma:horoheightminimum} guarantees that the geodesic $\gamma_n$ joining them passes through the gauge unit sphere $\mathbb S$ in at least one point at a horoheight bounded below by some $h_0>0$. Let $\Sph' = \{ s\in \mathbb S : \height_\infty(s)\geq h_0\}$, noting that it is a compact subset of $\Hyp$, and let $s_0 \in \gamma \cap \Sph'$ be some intersection point, possibly not unique due to non-constant curvature. 

For simplicity of notation, assume that $n_0=0$. Iterating further, it may happen that $\norm{M_{n_0+1}x'}<\rad(K)^{-1/2}$, in which case \eqref{eq:magicinversion} again pushes the iterates of $x$ and $x'$ apart. Thus, for some $n_1>n_0$ we will again have that $(x_{n_1},x'_{n_1})$ are widely-spaced, with $\gamma_{n_1}$ passing through the set $\Sph'$. Since $s_0\in \gamma$, we know $M_{n_1} s_0\in \gamma_{n_1}$, but it may no longer be the case that $M_{n_1} s_0\in \Sph'$. Instead, note that the mapping $M$ moves points along $\gamma$ with some translation distance $\ell_M$, and that $M_{n_1}$ preserves this distance, so that $M_{n_1} M^j s_0$, $j\in \mathbb{Z}$, are a sequence of points along $\gamma_{n_1}$ separated by a distance $\ell_M$. In particular, there is some $j_1\in \N$ such that $M_{n_1} M^{j_1} s_0\in \overline{N_{\ell_M}\Sph'}$. Continuing to iterate, we obtain an infinite sequence of increasing indices $n_i$ where $(x_{n_i},x'_{n_i})$ are widely spaced and corresponding powers $j_i$ satisfying  $M_{n_i} M^{j_i} s_0\in \overline{N_{\ell_M}\Sph'}$. By Lemma \ref{lemma:properaction}, there are finitely many elements of $\Mod$ sending $s_0$ into $\overline{N_{\ell_M}\Sph'}$, so we must have that for some $i_1<i_2$ we have $M_{n_{i_1}}M^{j_{i_1}}=M_{n_{i_2}}M^{j_{i_2}}$. Thus, $M^{-1}_{n_{i_2}}M_{n_{i_1}}=M^{j_{i_2}}M^{-j_{i_1}}$. We conclude that $M^{-1}_{n_{i_2}}M_{n_{i_1}}x=M^{j_{i_2}}M^{-j_{i_1}}x=x$, giving $T^{n_{i_1}}x = T^{n_{i_2}}x$, so $x$ has a periodic expansion, as desired.

\end{proof}
\end{lemma}
\section{Loxodromic fixed points}
\label{sec:loxodromic}

The classical Euler-Lagrange theorem characterizes repeating CFs as the quadratic surds, while our Theorem \ref{thm:generalmain} instead characterizes them as the fixed points of loxodromic matrices. In this section, we first reconcile the two viewpoints in the real and complex cases, and then study the higher-dimensional cases.

\subsection{Real and complex cases}\label{sec:backwardscompatability} 
The following theorem can be proven using the Euler-Lagrange theorem. We give a direct proof that avoids CF theory. 

\begin{thm}\label{thm:backwards}
    A number $x\in \R$ is a quadratic surd if and only if it is a fixed point of a loxodromic transformation in $SL(2,\Z)$.
\begin{remark}
In $SL(2,\Z)$, loxodromic transformations coincide with hyperbolic transformations.
\end{remark}

\begin{proof}    
    Suppose first that $M=\begin{pmatrix} A & B\\ C & D \end{pmatrix}$ is loxodromic. Let $x_1$ and $x_2$ be the fixed points of $M$ under the linear-fractional action on $\hat \R$, so that $\frac{Ax_i+B}{Cx_i+D}=x_i$ and $x_1\neq x_2$. Then each $x_i$ satisfies the equation $Cx^2 + (D-A)x-B=0$, so each $x_i$ is a solution to a non-degenerate quadratic equation over $\Z$.

    Conversely, consider a non-degenerate quadratic equation $ax^2+bx+c=0$ with $a,b,c\in \Z$. Solutions to this equation are irrational, so the discriminant $\Delta=b^2-4ac$ a non-square integer. We seek a matrix $M\in SL(2,\Z)$ whose fixed points are the solutions of the given quadratic equation. The fixed points of $M$ satisfy the equation $Cx^2 + (D-A)x-B=0$ and multiples thereof, so we need to have $C=\lambda a$, $D-A=\lambda b$, and $-B=\lambda c$ for some $\lambda\neq 0$. We will, in fact, demand that $\lambda \in \N$. 
    
    We require $M\in SL(2,\Z)$, so we need
    $1=AD-BC = A(\lambda b + A) + \lambda^2 a c$.
    That is, we need $A=(-\lambda b \pm \sqrt{\lambda^2 \Delta + 4})/2$. Because $\Delta$ is non-square, the Pell's equation $\mu^2 \Delta +1 =n^2$ has a non-trivial solution $(\mu,n)\in \Z^2$, see \cite[Thm.~34.1]{silvermanfriendly}. Fix one solution with $\mu > 0$ and take $\lambda=2\mu$. The resulting choices of $A,B,C,D$ satisfy $\det M = 1$ and have the desired fixed points. Lastly, one observes that matrices in $SL(2,\Z)$ with exactly two fixed points are loxodromic. Indeed, one may send the two points to $0$ and $\infty$ by an element of $SL(2,\R)$ so that $M$ is conjugate to $\begin{pmatrix} \alpha & 0\\ 0 & \alpha^{-1} \end{pmatrix}$ for some $\alpha\neq \pm 1$.
\end{proof}
\end{thm}

To generalize to the complex setting, we require a solution to Pell's Equation:

\begin{thm}[Oswald \cite{oswald2015diophantine}]
Let $d\in \N$ and $\Delta\in \Z[\ii\sqrt{d}]$ non-square. Then the equation $x^2+\Delta y^2=1$ has infinitely many solutions $x,y\in \Z[\ii\sqrt{d}]$.
\begin{proof}[Sketch of proof]
We briefly describe Oswald's method for completeness. To begin with, Oswald shows that irrational values $z\in \C$ can be approximated by infinitely many rationals $p/q$ with $p,q\in \Z[\ii\sqrt{d}]$ by $|z-p/q| \le C/|q|^2$ for some constant $C>0$. For certain values of $d$ this can be proven by CF theory, but Oswald uses results from the geometry of numbers instead. Then, after writing $x^2-\Delta y^2=(x-y\sqrt{\Delta })(x+y\sqrt{\Delta })$, one can choose $x/y$ to be one of these good rational approximations for $\sqrt{\Delta }$, guaranteeing that $x^2-\Delta y^2$ belongs to a finite set of integers. These solutions can be manipulated to find a solution to $x^2-\Delta y^2=1$. 
\end{proof}
\end{thm}

As in the real case, $x\in \C$ is a quadratic surd over $\Z[\ii \sqrt{d}]$ if it is a solution to a quadratic polynomial with coefficients in $\Z[\ii \sqrt{d}]$ but not to a linear one.

It should be emphasized that although the set of all roots of polynomials in $\mathbb{Z}[x]$ and all roots of polynomials in $\mathbb{Z}[\ii][x]$ are the same, it is not true that they will necessarily have the same minimal polynomial degree. For instance, $\sqrt{1+\ii}$ has minimal polynomial is $x^4-2x^2-2$ in $\mathbb{Z}[x]$, and $x^2-(1+\ii)$ in $\mathbb{Z}[\ii][x]$. It can also be checked that $[1;\overline{-2\ii,2}]$ is a periodic complex CF expansion for $\sqrt{1+\ii}$.

We are now ready to state the complex result, whose proof is identical to the proof of Theorem \ref{thm:backwards}, with the adjustments mentioned above.
\begin{thm}
Let $x\in\mathbb{C}$, and let $d\in \mathbb{N}$. Then $x$ is a quadratic surd over $\mathbb{Z}[\ii\sqrt{d}][x]$ if and only if it is the fixed point of a hyperbolic transformation in $\operatorname{SL}(2,\mathbb{Z}[\ii\sqrt{d}])$.
\end{thm}

\subsection{Clifford algebra formalism}\label{sec:Clifford}

Real and complex CFs are commonly encoded using matrices in $SL(2,\R)$ and $SL(2,\C)$, respectively. For Iwasawa CFs with underlying space $\R^n$, including quaternionic and octonionic CFs, one can encode the CFs using 2-by-2 Clifford matrices over the appropriate Clifford algebra. We now recall this formalism, following Ahlfors \cite{ahlfors1985fixed} who builds on Vahlen \cite{Vahlen1902}.

The Clifford algebra $A_n$, $n\ge 0$, is an associative algebra over the reals generated by elements $e_1, \dots e_n$, which satisfy the relations $e_i^2 =-1$ and $e_ie_j=-e_je_i$ for $1\le i,j\le n$. Elements of $A_n$ can be written uniquely as sums
\[
\sum_{I\subset \{1,2,\dots, n\}} a_I \prod_{i\in I} e_i
\]
where the product goes from left to right in increasing order of index. For simplicity, we write $a_0$ instead of $a_{\emptyset}$ and ignore set-notation in subscripts, so that $a_{12}=a_{\{1,2\}}$. Thus, $A_2$ consists of elements of the form $a_0+a_1e_1+a_2e_2+a_{12}e_1e_2$, with real coefficients. 

The \emph{involution} operation on $A_n$ takes $a$ to $a'$, which is defined by replacing all $e_i$ with $-e_i$. The \emph{reversion} operation on $A_n$ takes $a$ to $a^*$, which is defined as the element with the product of $e_i$'s all reversed in order. Then $\overline{a}$ is defined as $\overline{a}={a'}^*$.

Identify $\R^n$ with elements of $A_{n-1}$ of the form $a_0+a_1e_1+a_2e_2+\dots+a_{n-1}e_{n-1}$.
The Clifford group $\Gamma_n$ is then the multiplicative subgroup of $A_{n-1}$ generated by non-zero elements of $\R^n$. On $\Gamma_n$, we can define a real-valued function $\norm{a}^2=a\overline{a}$, satisfying $\norm{ab}=\norm{a}\norm{b}$. Inverses in $\Gamma^n$ are then given by $a^{-1}=\overline{a}/\norm{a}^2$. 

Now, a matrix \[
\begin{pmatrix}
a& b \\ c & d
\end{pmatrix}\]
is a Clifford matrix if it satisfies
\begin{enumerate}
\item $a,b,c,d\in \Gamma_n\cup \{0\}$
\item $ad^*-bc^*=1$
\item $ac^{-1}$ and $c^{-1}d$ belong to $\R^n$ if $c\neq 0$.
\item $db^{-1}$ and $b^{-1}a$ belong to $\R^n$ if $b\neq 0$.
\end{enumerate}
One shows that Clifford matrices form a group and act on the extended real vectors $\widehat{\R^n}=\R^n \cup\{\infty\}$ by fractional linear transformations in the usual way:
\[
\begin{pmatrix}
a& b \\ c & d
\end{pmatrix}= (ax+b)(cx+d)^{-1}.
\]
The inversion $z\mapsto -1/z$ and shifts $z\mapsto z+a$ on $\R^n\setminus\{0\} \subset \Gamma_n$ then correspond, respectively, to the familiar Clifford matrices $\begin{pmatrix}
0& -1 \\ 1 & 0
\end{pmatrix}$ and  $\begin{pmatrix}
1& a \\ 0 & 1
\end{pmatrix}.$
Indeed, according to Vahlen's theorem \cite{Vahlen1902} (proof sketched in \cite{ahlfors1985fixed}) the Clifford matrices acting on the hyperbolic upper half-space $\Hyp\subset \widehat{\R^{n+1}}$ are  (modulo $\pm \operatorname{Id}$) the orientation-preserving isometries of $\Hyp$.

We record a simple but useful result about 2-by-2 matrices over any algebra:

The primary challenge for us in working with Clifford matrices is that it is difficult to identify elements of $\Gamma_n$, and therefore difficult to detect whether an integer-valued matrix is a Clifford matrix. While we will not fully solve that problem here, we will show it is nevertheless tractable.

\begin{lemma}
\label{lemma:hiddenSymmetries}
Let $A$ be an algebra and $a,b\in A$ be inverse elements. Then,
\(
\begin{pmatrix}0&-1\\1&0\end{pmatrix}
\begin{pmatrix}1&-b\\0&1\end{pmatrix}
\begin{pmatrix}0&-1\\1&0\end{pmatrix}
\begin{pmatrix}1&-a\\0&1\end{pmatrix}
\begin{pmatrix}0&-1\\1&0\end{pmatrix}
\begin{pmatrix}1&-b\\0&1\end{pmatrix}=
\begin{pmatrix}a&0\\0&b\end{pmatrix}.
\)
\end{lemma}
\subsection{Loxodromic fixed points in $\R^3$}\label{sec:R3}

Here we will consider continued fractions in $\R^3$, with associated Clifford algebra 
\[
A_2 = \{a_0+a_1e_1+a_2e_2+a_{12}e_1e_2:a_0,a_1,a_2,a_{12}\in \R\}.
\]
We may in fact identify $A_2$ with the quaternions, using $e_1=\ii$, $e_2=\jj$, and $e_1e_2=\kk$. All the operations on $A_2$ act as their usual counterparts do on the quaternions.

Our continued fractions will have digits in $\Z^3=\{a_0+a_1e_1+a_2e_2:a_0,a_1,a_2\in\Z\}$, and thus we will make notations for 
\[
A_2(\Z)=\{a_0+a_1e_1+a_2e_2+a_{12}e_1e_2:a_0,a_1,a_2,a_{12}\in \Z\}
\]
and $\Gamma_3(\Z)$, which are products of non-zero elements of $\Z^3$. We will let $SL(2,\Z^4)$ denote the set of Clifford matrices whose coordinates all lie in the restricted set $\Gamma_3(\Z)\cup \{0\}$. 

We now show that $SL(2,\Z^4)$ is the generalized modular group $\Mod$ of a 3D continued fraction algorithm given by the data $\X=\R^3$ seen as the quaternions with no $\kk$ component, $\iota(z)=-1/z$, $\mathcal{Z}=\Z^3$, and $K$ the unit cube centered at the origin (note that we will not use $K$ here).

\begin{lemma}
\label{lemma:modulargroup}
Let $\Mod$ be the group of Clifford matrices generated by the inversion $(x,y,z)\mapsto \frac{(-x,y,z)}{x^2+y^2+z^2}$ and translations in $\Z^3$. Then $\Mod=SL(2,\Z^4)$.
\begin{proof}
The generators of $\Mod$ are encoded by the Clifford matrices $\begin{pmatrix}0&-1\\1&0\end{pmatrix}$ and $\begin{pmatrix}1&a\\0&1\end{pmatrix}$, for $a\in \Z^3$. Thus, $\Mod$ is a subgroup of $SL(2,\Z^4)$.

Conversely, we suppose $\begin{pmatrix}a&b\\c&d\end{pmatrix}\in SL(2,\Z^4)$ with $c\neq 0$. We will apply the Euclidean algorithm in $\R^3$ to the first column of this matrix. 

If $\norm{c}>\norm{a}$, then we replace the matrix with 
\[
\begin{pmatrix}A &B\\C&D\end{pmatrix}= \begin{pmatrix}0&-1\\1&0\end{pmatrix} \begin{pmatrix}a&b\\c&d\end{pmatrix}
\]
so that $\norm{C}=\norm{a}<\norm{c}=\norm{-c}=\norm{A}$.

If $\norm{c}\le \norm{a}$, note that $ac^{-1}\in \R^3$ by definition of Clifford matrices. We compute the nearest integer $a_0=\round{ac^{-1}}\in\Z^3$ and replace the matrix with $\begin{pmatrix}A&B\\C&D\end{pmatrix}= \begin{pmatrix}1&-a_0\\0&1\end{pmatrix} \begin{pmatrix}a&b\\c&d\end{pmatrix}$. We then have that $\norm{ac^{-1}-a_0}<1$, so that $\norm{A}=\norm{a-a_0c}<\norm{C}$ and therefore $\norm{A}<\norm{C}$. 

We repeat the two previous paragraphs above, replacing the matrix with another matrix within $\Mod$. At each step, either the maximum norm of the left column is the same (if we applied the inversion matrix) or smaller than before (if we applied the translation matrix), and we apply the translation matrix at every other step. Since the norms of elements of $\Gamma(\Z)$ belong to $\Z_{\ge 0}$, which is discrete, this algorithm terminates in finite time at a matrix with $c=0$.

We have thus reduced the situation to a Clifford matrix with $c=0$, $bd^{-1}$ a vector, and $ad^*=1$. Left-multiplying by $\begin{pmatrix}1&-bd^{-1}\\0&1\end{pmatrix}$, we are reduced to a diagonal matrix with $ad^*=1$.

Since $\norm{a}\norm{d^*}=1$, we must have $\norm{a}=1$, so $a\in \{\pm 1, \pm \ii ,\pm \jj, \pm \kk\}$. We will show that it is possible to obtain all 8 such matrices inside $\Mod$. For $a\in \{\pm 1, \pm \ii, \pm \jj\}$, each choice of $a$ is invertible inside $\Gamma(\Z)$, so the desired decomposition is given by Lemma \ref{lemma:hiddenSymmetries}. For $a=\pm \kk$, we cannot directly use $\kk$ as a digit in Lemma \ref{lemma:hiddenSymmetries}, but we have that $\kk=\ii\jj$, so we can write it as a product of generators by first doing so with $\ii$ and $\jj$ as above and multiplying the resulting diagonal matrices.
This shows that $SL(2,\Z^4)\subset \Mod$, completing the proof.
\end{proof}
\end{lemma}

Categorizations of quaternionic matrices into elliptic, parabolic, and loxodromic types are somewhat complicated. We present here the simplest categorization the authors are aware of, due to Parker and Short \cite{parker2009conjugacy}.

Given an arbitrary matrix $A=\begin{pmatrix} a & b \\ c & d \end{pmatrix}$ with quaternionic entries, define the following quantities:
\begin{align*}
    \alpha &= \norm{a}^2\norm{d}^2+\norm{b}^2\norm{c}^2-2 \Re(a\overline{c}d\overline{b})\\
    \beta &= \Re\left((ad-bc)\overline{a}+(da-cb)\overline{d}   \right)\\
    \gamma &= \norm{a+d}^2+2\Re(ad-bc)\\
    \delta &= \Re(a+d)
\end{align*}
Parker and Short define the group $SL(2,\mathbb{H})$ to be all such matrices with $\alpha=1$. These matrices are not necessarily Clifford matrices, but do contain $SL(2,\Z^4)$ as a subgroup, which can be easily seen by comparing the corresponding M\"{o}bius transformations.  For matrices in $SL(2,\mathbb{H})$, two more quantities are needed:
\begin{align*}
\sigma &= \begin{cases}
cac^{-1}d-cb, & c\neq 0,\\
bdb^{-1}a, & c=0,b\neq 0,\\
(d-a)a(d-a)^{-1}d, & b=c=0, a\neq d,\\
a\overline{a}, & b=c=0,a=d,
\end{cases}\\
\tau &= \begin{cases}
cac^{-1}+d, & c\neq 0,\\
bdb^{-1}+a, & c=0, b\neq 0\\
(d-a)a(d-a)^{-1}+d, & b=c=0,a\neq d,\\
a+\overline{a}, & b=c=0, a=d,
\end{cases}
\end{align*}
which function as the quaternionic determinant and trace. Finally a matrix is simple if it is conjugate to a matrix with real entries and $k$-simple if it can be expressed as a product of $k$ simple matrices and no fewer. All matrices in $SL(2,\mathbb{H})$ are at most $3$-simple.

We note that the generators for $SL(2,\Z^4)$ belong to $SL(2,\mathbb{H})$ and thus the former is a discrete subgroup of the latter. We can therefore use the following characterization of matrices by Parker and Short to categorize matrices in $SL(2,\Z^4)$.

\begin{prop}[Theorem 1.4, \cite{parker2009conjugacy}]\label{prop:parker}
Suppose $A\in SL(2,\mathbb{H})$.
\begin{enumerate}
\item If $\sigma=1$ and $\tau\in \mathbb{R}$, then $A$ is $1$-simple and 
\begin{enumerate}
\item if $0\le \delta^2<4$ then $A$ is elliptic,
\item if $\delta^2=4$, then $A$ is parabolic,
\item if $\delta^2>4$, then $A$ is loxodromic.
\end{enumerate}
\item If $\beta=\delta$ and either $\tau\not\in \R$ or $\sigma\neq 1$ then $A$ is $2$-simple and
\begin{enumerate}
\item if $\gamma-\delta^2<2$, then $A$ is elliptic,
\item if $\gamma-\delta^2=2$, then $A$ is parabolic,
\item if $\gamma-\delta^2>2$, then $A$ is loxodromic.
\end{enumerate}
\item If $\beta\neq \delta$, then $A$ is $3$-simple loxodromic.
\end{enumerate}
\end{prop}

We note that if $A$ is a Clifford matrix, then the notion of elliptic, parabolic, and loxodromic here all have the same meaning in Parker and Short's work as they do in ours.

\begin{example}Cao, Parker, and Wang \cite{cao2004classification} provided a classification for quaternionic matrices similar to Proposition \ref{prop:parker} above, but for those matrices which preserve the unit ball. Their work can be translated into the situation we are studying by means of the matrix $\begin{pmatrix}\kk/2 & 1/2 \\ -1 & -\kk \end{pmatrix}$, which maps the unit ball bijectively onto the upper-half space. Example 4.3 in \cite{cao2004classification} shows that the matrix $\begin{pmatrix}1 & (\ii+\jj)/2\\ -2\ii-2\jj & 3 \end{pmatrix}$ is loxodromic and fixes the points $\frac{-1\pm \sqrt{3}}{4} (\ii+\jj)$. These both have periodic CF expansions:
\begin{align*}
\frac{-1+\sqrt{3}}{4}(\ii+\jj)&= [3(\ii+\jj),\overline{-2(\ii+\jj),4(\ii+\jj)}]\\
\frac{-1-\sqrt{3}}{4}(\ii+\jj)&= [-1(\ii+\jj),\overline{2(\ii+\jj),-4(\ii+\jj)}].
\end{align*}
\end{example}

\begin{example}
In the opposite direction, let us consider a matrix that fixes a known point with a periodic continued fraction expansion and show that it is indeed loxodromic. We will consider the matrix
\[M=
\begin{pmatrix}
0 & -1 \\
1 & -a 
\end{pmatrix}
=
\begin{pmatrix}
0 & -1 \\
1 & 0 
\end{pmatrix}
\begin{pmatrix}
1 & -a \\
0 & 1 
\end{pmatrix},
\]
where $a$ is an element of the digit set for the three-dimensional continued fraction, satisfying $\|a\|>(4+\sqrt{3})/2$. As in Pringsheim's Theorem, the size requirement guarantees that the matrix maps $K$ fully into $K$, guaranteeing convergence of the periodic CF $[\overline{a}]$. Furthermore, it is easy to see $M$ fixes $[\overline{a}]$. Now we employ Proposition \ref{prop:parker} above. To begin with, we quickly see that $\alpha = 0|a|^2+1-0=1$, so the matrix is in $SL(2,\Hyp)$. Also $\beta = \Re((0\cdot -a+1)0+(-a\cdot 0 +1) \overline{-a}) = -\Re(a)$ and $\delta =\Re(0-a)=-\Re(a)$, so $\beta=\delta$. In addition, $\sigma = 1\cdot 0 \cdot 1^{-1} \cdot-a+1 = 1$ and $\tau = 1\cdot 0 \cdot 1^{-1} - a=-a$. If $a$ is real, so that we are in case $1$ of the proposition, we must have $|a|\ge 3$, and hence $\delta^2 >4$, so our matrix is 1-simple loxodromic. If $a$ is not real, we are in case 2 of the proposition and here $\gamma=|0-a|^2+2\Re(0\cdot -a +1) = |a|^2+2$. Thus $\gamma-\delta^2 = |a|^2+2-\Re(a)^2$, and since $|a|^2> \Re(a)^2$ since $a$ is not real, we have that $\gamma-\delta^2>2$ and so $A$ is 2-simple loxodromic, as desired.
\end{example}

\subsection{Loxodromic fixed points in $\R^4$}\label{sec:R4}
We now consider continued fractions in $\R^4$, with associated Clifford algebra
\[
A_3 = &\{a_0+a_1e_1+a_2e_2+a_3e_3+a_{12}e_1e_2\\\notag&+a_{13}e_1e_3+a_{23}e_2e_3+a_{123}e_1e_2e_3:a_I\in \R\}.
\]
As before, we identify $\R^4$ with the subspace $\{a_0+a_1e_1+a_2e_2+a_3e_3: a_i \in \R\}$. Lastly, denote by $\Z^8$ the subgroup of $A_3$ where all coefficients are in $\Z$.

\begin{remark}
CF algorithms on $\R^4$ include the quaternionic CF algorithms (see e.g.~\cite{lukyanenko2022convergence} and references therein), just as CF algorithms on $\R^2$ include complex CF algorithms. We will not explicitly discuss quaternionic CFs here to avoid confusion with the Clifford algebras involved. Our choice of lattice and inversion below will reflect our interest in quaternionic CFs.
\end{remark}

Because CF algorithms with lattice $\Z^4$ are not proper, we will be interested in CF algorithms with the Hurwitz integer lattice $\mathcal{H}=\Z^4+h \Z^4$, where $$h=(1+e_1+e_2+e_3)/2.$$ Denote by $\Gamma_\mathcal{H}$ the multiplicative group in $A_3$ generated by $\mathcal{H}\setminus \{0\}$. We first show that $\Gamma_\mathcal{H}$ is discrete. 

\begin{lemma} The group $\Gamma_\mathcal{H}$ is discrete; indeed, 
$\Gamma_\mathcal{H}\cup\{0\}\subset \mathbb{Z}^8 +h\mathbb{Z}^8$.
\end{lemma}

\begin{proof}
It suffices to show that $\mathbb{Z}^8 +h\mathbb{Z}^8$ is closed under multiplication. Write $a \equiv b$ if $a-b\in \Z^8$. It is straightforward to check that $h^2\equiv h$; that for each $i$ one has $e_i h \equiv h e_i$; and thus that $h\alpha \equiv \alpha h$ for any $\alpha \in \Z^8$. One then computes, for $\alpha_1+h\beta_1, \alpha_2+h\beta_2\in \Z^8 + h\Z^8$,
\begin{align*}
(\alpha_1+h\beta_1)(\alpha_2+h\beta_2)
& \equiv h\beta_1\alpha_2+\alpha_1h\beta_2+h\beta_1h\beta_2\\
&\equiv h\beta_1\alpha_2 + (h\alpha_1+\gamma_1)\beta_2+h(h\beta_1+\gamma_2)\beta_2 \\
&\equiv h (\beta_1\alpha_2+\alpha_1\beta_2+\beta_1\beta_2 +\gamma_2\beta_2),
\end{align*}
where $\gamma_1=\alpha_1h-h\alpha_1$, $\gamma_2=\beta_1 h-h\beta_1$ are both in $\Z^8$. This shows the product $(\alpha_1+h\beta_1)(\alpha_2+h\beta_2)$ is again in $\Z^8+h\Z^8$.
\end{proof}

\begin{lemma}
Let $\mathcal M$ be the group of matrices generated by the inversion $(x,y,z,w)\mapsto \frac{(-x,y,z,w)}{x^2+y^2+z^2+w^2}$ and translations in $\mathcal H$. Then $\mathcal M$ is the group $SL(2,\Gamma_\mathcal{H})$ consisting of Clifford matrices with entries in $\Gamma_\mathcal{H}$.
\begin{proof}
As in the proof of lemma 4.6, it suffices to write  the diagonal matrices 
$\begin{pmatrix}
a & 0 \\
0 & a^{-1} 
\end{pmatrix}$ in terms of the generators of $\Mod$, where $a\in \Gamma_\mathcal{H}$ has $\norm{a}=1$. The group $\Gamma_\mathcal{H}$ is the multiplicative group generated by the non-zero elements of $\mathcal H$. Since norm is multiplicative and non-zero elements of $\mathcal{H}$ have norm at least 1, it suffices to consider $a\in \Gamma_\mathcal{H}$ that can be written as products of unit-norm elements of $\mathcal H$. Furthermore, it suffices to simply consider unit-norm elements of $\mathcal H$. For unit-norm elements of $\mathcal H$, lemma \ref{lemma:hiddenSymmetries}.
\end{proof}
\end{lemma}

We conclude by mentioning that identifying elliptic, parabolic, or loxodromic matrices in $SL(2,\Gamma_\mathcal{H})$ is quite involved, see Wang et al \cite{wang1998note}.

\section{A note on quaternionic identities}
\label{sec:identities}
In working with this paper, we observed a generalization of the identity $\ii+\frac{1}{\ii}=0$  over the quaternions, which is easily verified:

\begin{prop}
Let $p=\ii+\jj$ and $q=\ii+\jj+\kk$. Then $p+\frac{1}{p+1/p}=0$ and \[q+\cfrac{1}{q+\cfrac{1}{q+\cfrac{1}{q+\cfrac{1}{q}}}}=0.\] The identities remain true if the coefficients of $\ii, \jj,\kk$ are negated or permuted.
\end{prop}

No analogous identities for this inversion hold in higher dimensions. That is, fix $d>0$ and consider the space $\R^d$ with the inversion $\iota(x_1, \ldots, x_d)=\frac{(x_1, -x_2, \ldots, -x_d)}{x_1^2+\ldots+x_d^2}$. Let $a=(0,1,1, \ldots, 1)\in \R^d$.  The proposition states that for the complex numbers ($d=2$) we have $a+[a]=0$; in $\R^3$, we have $a+[a,a]=0$; and in $\R^4$, we have $a+[a,a,a,a]=0$.

\begin{prop}
For $d>4$, $a+[a,\ldots, a]\neq 0$ for arbitrary-length continued fractions.
\begin{proof}
Consider a CF $a+[a, \ldots, a]$ of length $n$. Observe that the first digit is 0, and the remaining digits are all equal to one another; we will call this value $x_n$.  Adding an additional CF digit to $a+[a,\ldots, a]$ changes $x_n$ by
$$x_{n+1}=1+\frac{-x_n}{\Norm{(0, x_n, \ldots, x_n)}^2}=1+\frac{-1}{(d-1)x_n}.$$
For $d=5$, one shows directly that $x_n=\frac{n+1}{2n}\neq 0$. When $d\geq 5$, this dynamical system has fixed points $x_\pm =\frac{1\pm \sqrt{1-4/(d-1)}}{2}$. For $d>5$, the interval $I=[x_-, x_+]$ is non-degenerate, and is mapped to itself under the dynamical system. In particular, since we have $x_1=1\in I$, we cannot have that $x_n=0$ for any $n$. 
\end{proof}
\end{prop}

\bibliographystyle{plain}
\bibliography{bibliography}

\end{document}